\documentclass[10pt]{article}%
\usepackage{amsmath}%
\usepackage{amsfonts}%
\usepackage{amssymb}%
\usepackage{graphicx}
\providecommand{\U}[1]{\protect\rule{.1in}{.1in}}
\newtheorem{theorem}{Theorem}

\newtheorem{conjecture}[theorem]{Conjecture}

\newtheorem{lemma}[theorem]{Lemma}

\newenvironment{proof}[1][Proof]{\noindent\textbf{#1.} }{\ \rule{0.5em}{0.5em}}

\newcommand{\RR}{{\mathbb R}}

\newcommand{\DD}{{\Delta}}

\newcommand{\TT}{{\mathbb T}}

\newcommand{\cpc}{{\mbox{Cap}}}

\newcommand{\TTT}{{T}}

\newcommand{\DDo}{{\overline{\Delta}}}
\begin{document}
\title{Capacity of shrinking condensers in the plane}
\author{Nicola Arcozzi}
\date{26/8/11}
\begin{abstract}
We show that the capacity of a class of plane condensers is comparable to the capacity of corresponding 
``dyadic condensers''. As an application, we show that for plane condensers in that class the capacity blows up as
the distance between the plates shrinks, but there can be no asymptotic estimate of the blow-up.
\end{abstract}
\maketitle
\section{Introduction}
Let $\Omega$ be an open region in the complex plane and let $E$ and $K$ be disjoint subsets of 
$\overline{\Omega}$, with $F$ closed
and $K$ compact. The \it capacity \rm of the \it condenser \rm  $(F,K)$  in $\Omega$ is
$$
\cpc_\Omega(F,K)=\inf\left\{\|\nabla u\|_{L^2(\Omega)}^2:\ u\ge1\ \mbox{on}\ K,\ \ u\le0\ \mbox{on}\ F\right\}.
$$
The sets $F$ and $K$ are the \it plates \rm of the condenser.
The infimum is taken over functions $u$ which are $C^1$ in $\Omega$ and continuous on $\Omega\cup F\cup K$. 
The capacity of a condenser, a notion arising in electrostatics, became part of mainstream Potential Theory in 
1945 with the foundational articles by Polya and Szeg\"o \cite{PSz} and Szeg\"o \cite{Sz}, 
where the case of $\RR^n$, $n\ge2$, is considered. Condenser capacity has since become an important and useful 
notion in mathematics \it per se \rm; for instance, in the theory of conformal ($n=2$) and quasiconformal 
($n\ge2$) 
mappings and, more generally, in Geometric Measure Theory on metric spaces. A class of problems 
in the field deals with estimates of capacity for condensers the plates of which undergo geometric 
transformations of some kind: rigid movement, for instance, or degeneration of one plate. 
Here, we will consider 
some condensers in which the space between the plates shrinks. Intuition suggests that the capacity of such 
condensers must blow up: we will see that this is true, but in a weak sense only. 

Problems of this kind have been considered before in the literature. 
In \cite{LL}, the plates of the condenser are identical discs getting closer.
In \cite{Karp},  the case
of concentric circular arcs, which are symmetric with respect to the real axis, is considered.
Other articles deal with different families of condensers and give rather precise estimates of how their 
capacity blows up as the distance between the plates vanishes. In this article, we consider a condenser
in which one plate is a disc of radius increasing to one, while the other is a compact subset of the unit circle, 
subject to a constraint on its capacity only.

\medskip

Let $\DD$ be the unit disc in the complex plane and let 
$\TT$ be its boundary, $\DDo=\DD\cup\TT$. We denote by $\DD(z,r)$ the disc of radius $r$ centered at $z$ and
by $\DDo(z,r)$ its closure.
Given $E,F\subset\overline{\DD}$, closed and disjoint, 
the  capacity of the  condenser  $(E,F)$ in $\DD$ is
$$
\cpc(E,F)=\inf\left\{\|\nabla u\|_{L^2(\DD)}^2:\ u\ge1\ \mbox{on}\ E,\ u\le0\ \mbox{on}\ F\right\},
$$
where, for points $\zeta\in E\cup F\setminus\DD$, we ask for the existence of 
$\lim_{r\to1}u(r\zeta)\in\RR\cup\left\{+\infty\right\}$.
In this article, we define the \it capacity \rm of $E\subseteq\TT$ to be
$$
\cpc(E):=\cpc(E,\DDo(0,1/2)).
$$
The quantity $\cpc(E)$ is comparable with the logarithmic capacity of $E$. 
We are here interested in the behavior of $\cpc(E,\DDo(0,r))$ as $r\to1$, when 
$E\subseteq\TT$ is closed and has positive capacity (if $\cpc(E)=0$, then  
$\cpc(E,\DDo(0,r))=0$ for all positive $r<1$).
\begin{theorem}\label{positive}
Let $E\subseteq\TT$ be closed and $\cpc(E)>0$. Then,
$$
\lim_{r\to1^-}\cpc(E,\DDo(0,r))=+\infty.
$$
\end{theorem}
It is not possible to find asymptotic estimates for the rate of convergence.
\begin{theorem}\label{negative}
Let $0<\epsilon<\epsilon_0$, with $\epsilon_0$ small enough. Then, there is $C(\epsilon)>0$ such that for all
$r\in(0,1)$ there is a closed subset  $E=E_{\epsilon,r}$ of $\TT$ with $\cpc(E)=\epsilon$,
yet $\cpc(E,\DDo(0,r))\le C(\epsilon)$.
\end{theorem}
Theorem \ref{negative} could also be deduced from Haliste's desymmetrization result in \cite{Hal}. In fact, 
we deduce it from an elementary, discrete desymmetrization inequality. Set desymmetrization was introduced in  
\cite{Du1}, and it has proved a powerful tool in potential theory.

If $E$ has full capacity, $\cpc(E)=\cpc(\TT)$, then $E=\TT$ and the problem of the rate reduces to an 
elementary calculation:
$$
\cpc(\TT,\DDo(o,r))=\frac{1}{\log r^{-1}}\sim(1-r)^{-1}\ \mbox{as}\ r\to1.
$$ 
It would be interesting to have an extension of Theorem \ref{positive} to the case $\epsilon_0=\cpc(\TT)$. 
\begin{conjecture}\label{conjecture}
$$
\inf\left\{\cpc(E,\DDo(0,r)):\ \cpc(E)=\cpc(\TT)-\delta\right\}\approx\frac{\cpc(\TT)-\delta}{r+\delta}.
$$
\end{conjecture}
The conjecture, that is, is that the right ```scale'' governing the asymptotics of capacity for the condensers 
considered in this article
is not given 
by a small set capacity, but by 
the small amount by which the closed set $E$ fails to have full capacity (hence, to be the 
full boundary $\TT$). 
We offer below some evidence in favor of the conjecture. 
If the conjecture were true, Theorem \ref{negative} would also hold 
without the assumption that $\epsilon$ 
be ``small enough''.

The method we employ in proving Theorems \ref{positive} and \ref{negative} seems to be new in the 
context of condensers, and it might be useful in tackling similar problems. We will consider first, in 
Section \ref{sectwo}, a discrete, ``toy'' version of the original estimates on a dyadic tree. 
The discrete problem turns out to be much easier to solve. On trees scaling arguments are natural and lead to 
precise formulas; the boundaries of ``connected regions'' are rather trivial and condensers are much simpler 
objects; more important: there is a precise recursive algorythm to compute the capacity of a set. In the tree 
context, we will prove analogs of Theorem \ref{positive} and of a sharper version of Theorem \ref{negative}. 
Then, we show that the relevant quantities (capacities of sets and condensers, distance between the plates) 
can be transfered from the discrete setting to the disc 
setting and back, with estimates from above and below. In Theorem \ref{positive}, we use the fact that, 
essentially, a unique function which is harmonic on the tree encodes the extremals for all condensers obtained 
by shrinking the space between the plates.
In Theorem \ref{negative}, the advantage is that a recursive argument on the tree, which is wholly precise, 
gives a good estimate for a condenser capacity in the disc: the loss of precision happens just ones, passing 
from the tree to the disc.

The idea of using potential theory on trees to solve problems in the continuous setting is not new. 
In \cite{BePe}, Benjamini and Peres showed that logarithmic and ``tree'' capacity of a subset on the real 
line are comparable, and used this fact to explain the transience-recurrence dichotomy for a random walk on a tree 
in terms of classic logarithmic capacity. In \cite{ARSW2}, a different proof of the same 
result is given, and it is applied to the proof of a Nehari-type theorem for bilinear forms on the holomorphic 
Dirichlet space. In \cite{ARSW} it is shown in some generality (\it Ahlfors regular metric spaces, \rm 
non-linear potentials) that Bessel-type set capacities are equivalent to analogous set capacities on trees. 
The novelty here is that the equivalence between discrete and continuous setting is extended to the capacities 
of some \it condensers\rm.

\smallskip

This work was born from a question of Carl Sundberg, who asked me if something was known on the rate of 
convergence to infinity of the condensers considered in this article. It is a pleasure to thank him for the 
stimulating question and the organizers of the RAFROT $2010$ Conference in Portorico, where the question was 
posed to me and where I gave a first (wrong) answer. 
Thanks also go to D. Betsakos, for the useful comments and suggestions on a first draft of the paper.

\section{Capacities on trees}\label{sectwo}
\paragraph{Trees.} We start by recalling some basic facts about trees. Let $T$ be a dyadic tree with root $o\in T$. Each vertex of $T$ 
is linked by an edge to three other vertices, except for the root, who is linked to just two
vertices. The eventual edge between $x$ and $y$ is denoted by $[x,y]$ ($[x,y]=[y,x]$). 
A path $\Gamma_{x,y}$ between points $x,y\in T$ is defined as a sequence of edges:
$\Gamma_{x,y}=\left\{[x_{j-1},x_j]:\ j=1,\dots,n\right\}$ with $x_0=x$ and $x_n=y$ 
($\Gamma_{x,y}=\Gamma_{y,x}$: we do not consider oriented paths). \it The \rm path 
$\Gamma_{x,y}=:[x,y]$ in which no edge appears more than once is the \it geodesic \rm 
between $x$ and $y$. With slight abuse, we can consider $[x,y]$ as a subset of $T$. We 
identify the edge $[x,y]$ with the geodesic $[x,y]$ between neighboring points and pose 
$[x,x]=\left\{x\right\}$. The  \it natural distance \rm between $x,y\in T$ is $d(x,y)=\sharp[x,y]-1$. 
We write $d(x)=d(x,o)=\sharp[x,o]-1$. Given $x,y\in T$, we 
introduce the partial order: $x\le y$ if $x\in[o,y]$. For each $x\in T$ there are two neighbours
$x_+$ and $x_-$ which follow $x$ in the partial order. We say that $x_\pm$ are the 
\it children \rm of $x$ and that $x=:x_\pm^{-1}$ is their \it parent\rm.

We also consider \it half-infinite geodesics \rm $\gamma\subset T$ starting at $x\in T$, which 
might be defined as unions of geodesics $[x,x_n]\subset[x,x_{n+1}]$, with $d(x,x_n)\to+\infty$: 
$\gamma=\cup_{n\ge0}[x,x_n]$. The set of the half-infinite geodesics starting at $o$ is the
\it boundary \rm of $T$, denoted by $\partial T$. To avoid confusion, we consider $\partial T$ 
as a set of geodesics' labels: $\zeta\in \partial T$ labels the geodesic $P(\zeta)$. 
By extension, we write $P(x)=[o,x]$ when $x\in T$.
 If $\zeta\in\partial T$ and $x\in T$, we set $[x,\zeta):=P(\zeta)\setminus P(x^{-1})$:
the geodesic joining $x\in T$ and the boundary point $\zeta$. 
Let $\overline{T}=T\cup\partial T$.
Given $x\in T$, $S(x)\subseteq\overline{T}$ is the \it successor set \rm of $x$:
$S(x)=\left\{\zeta\in\overline{T}:\ x\in P(\zeta)\right\}$. We also set $T_x=S(x)\cap T=
\left\{y\in T:\ y\ge x\right\}$, the subtree of $T$ having root $x$. Note that $\partial T_x=\partial T\cap S(x)$
is the boundary of the rooted tree $T_x$.

 Given $\zeta\ne\xi\in\overline{T}$, let 
$\zeta\wedge\xi=\max(P(\zeta)\cap P(\xi))$, where the maximum is taken w.r.t. the partial
order. We introduce a new distance $\rho$ in $\overline{T}$,
$$
\rho(\alpha,\beta)=2^{-d(\alpha\wedge\beta)}-
\frac{1}{2}\left(2^{-d(\alpha)}+2^{-d(\beta)}\right).
$$
The metric space $(\partial T,\rho)$ is totally disconnected, perfect, compact. 
For $\xi,\zeta\in\partial T$, $\rho(\xi,\zeta)=2^{-d(\xi\wedge\zeta)}$,
while  $T$ is the set of the isolated points on $(\overline{T},\rho)$ and $\partial T$ is the
metric boundary of $T$ in $(\overline{T},\rho)$. In fact, we can identify $(\overline{T},\rho)$
with the metric completion of $(T,\rho)$ and $\partial T$ with the points which have been added 
to $T$ in order to make it complete w.r.t. the metric $\rho$.
The set $S(x)$ is then the closure of $T_x$ in $\overline{T}$.

We introduce a sum operator $I$ applying functions $\varphi:T\to\RR$ to functions $I\varphi:\overline{T}\to\RR$., 
$$
I\varphi(\zeta)=\sum_{x\in P(\zeta)}\varphi(x),\ \zeta\in \overline{T}.
$$
(We will consider $\varphi\ge0$, hence convergence of the series when $P(\zeta)$ is infinite causes no ambiguity).
Its formal adjoint $I^*$ acts on Borel measures $\mu$ on $\overline{T}$,
$$
I^*\mu(x)=\int_{S(x)}d\mu,\ x\in T.
$$
The ``Hardy'' operator $I$ on trees was first introduced, in connection with problems of classical 
function theory, in \cite{ARS}.

\paragraph{Tree capacities.}
Let $E$ be a closed subset of $\partial T$. Its capacity is
$$
\cpc^T(E)=\inf\left\{\|\varphi\|_{\ell^2}^2:\ \varphi\ge0,\ I\varphi\ge1\ \mbox{on}\ E\right\}.
$$
For $n\ge0$, integer, we consider a condenser capacity
$$
\cpc^T_n(E)=\inf\left\{\|\varphi\|_{\ell^2}^2:\ I\varphi\ge1\ \mbox{on}\ E,\ 
I\varphi(x)=0\ \forall x\ s.t.\ d(x)=n-1\right\}.
$$
We set $\cpc^T_0(E)=\cpc^T(E)$.

For each $x\in T$, the tree $T_x$ has boundary 
$\partial T_x=\partial S(x)$ and we can compute the capacity of sets $F\subseteq \partial T_x$ w.r.t. the root
$x$. If $E\subseteq\partial T$ and $E_x=E\cap \partial T_x$, it is clear from the definitions and 
the trivial topology of $T$ that
\begin{equation}\label{additivity}
\cpc^T_n(E)=\sum_{x:\ d(x)=n}\cpc^{T_x}(E_x).
\end{equation}
Here, $\cpc^{T_x}(E_x)$ is the capacity of $E_x$ in $T_x$ w.r.t. the root $x$.

Before we proceed, we give some basic properties of tree capacities. 
Proofs are sparse in the literature, or they are special cases of general theorems about 
capacities in metric spaces. A good source for the general theory is \cite{AH}. All properties are 
given a precise reference or proved in \S5 of \cite{ARSW} (for general trees and weighted potentials), and 
they are proved in \cite{ARS2} (in the dyadic case).
\begin{itemize}
\item[(a)] There exists a unique extremal function $h=\varphi$ for the definition of $\cpc^T(E)$.
It satisfies (i) $\|h\|_{\ell^2}^2=\cpc^T(E)$; (ii) $H=Ih\ge1$ on $E$, but for a set of null 
capacity.
\item[(b)] The function $h$ satisfies the algebraic relation $h(x)=h(x_+)+h(x_-)$ and $h(x)>0$ 
everywhere on $T$.
\item[(c)] There is a unique positive, Borel measure $\mu$ supported on $E$ (the equilibrium measure)
with the property that $h=I^*\mu$. Moreover, $\cpc^T(E)=\mu(E)$. As a consequence, $\cpc^T(E)=h(o)$
\item[(d)] $\lim_{x\to\partial T}H(x)=1$ $\mu-a.e.$.
\item[(e)] Capacities satisfy a recursive relation:
$$
\cpc^T_{x}(E_{x})=
\frac{\cpc^{T_{x_+}}(E_{x_+})+\cpc^{T_{x_-}}(E_{x_-})}{1+\cpc^{T_{x_+}}(E_{x_+})+\cpc^{T_{x_-}}(E_{x_-})}.
$$
\item[(f)] $\cpc^T(E)=h(o)$.
\end{itemize}
The capacity of the full boundary is $\cpc^T(\partial T)=1/2$.

\paragraph{Theorem \ref{negative} holds on trees.}
\begin{theorem}\label{negativeT}
$\forall\epsilon\in(0,1/2)\ \exists R>0\ \forall n\ \exists E\subset\partial T:\ 
\cpc^T(E)\ge\epsilon,$ but $\cpc^T_n(E)\le R$.
\end{theorem}
In fact, we can be more precise:
$$
R=\frac{\epsilon}{1-2\epsilon}.
$$
\begin{proof}
Consider a set $E$ such that $\cpc^T(E)=\epsilon$, a positive integer $n$, and suppose that, at each 
step $j=1,\dots,n$ the set splits in two copies having the same capacity. Namely, the set $E$ 
splits into two copies $E_+\subseteq\partial T_{o_+}$ and $E_-\subseteq\partial T_{o_-}$ having equal capacities, 
and so on, iterating. In the end we get, corresponding to the $2^n$ points
$x_1^n,\dots,x_{2^n}^n$ s.t. $d(x_j^n)=n$,  $2^n$ sets
$E_{1}^n\subset\partial T_{x_1},\dots,E_{2^n}^n\subset\partial T_{x_{2^n}}$ having equal capacity.

Let $e_n$ be the capacity of any of $E_j^n$ w.r.t. the root $x_j^n$. Indeed, $e_0=\epsilon$ and
$$
e_{n}=\frac{e_{n-1}}{2-2e_{n-1}}, 
$$
by (e). Iterating, we find
$$
\cpc^T_n(E)=2^ne_n=\frac{2^n\epsilon}{2^{n}-(2^{n+1}-2)\epsilon}\nearrow\frac{\epsilon}{1-2\epsilon}.
$$
To finish the proof, we must show that, for any given $\epsilon$ in $(0,1/2)$, there is a set $E$ 
having capacity $\cpc^T(E)=\epsilon$, which is the union of $2^n$ subsets having equal capacity, each lying in some 
$I(x_j^n)$, $1\le j\le 2^n$. This can be done if and only if we can find a subset $E_j^n$ of $I(x_j^n)$ such that
$\cpc^{T_{x_j^n}}(E_j^n)=e_n$; which (by obvious rescaling) is the same as finding a closed subset $F$ of 
$\partial T$ such that $\cpc^T(F)=e_n$. 
By induction and the fact that $\psi(t):=t/(2-2t)$ is a diffeomorphism
of $[0,1/2]$ onto itself, we have that $0<e_n<1/2$. Finally, it is easy, for each such $e_n$,
to produce a set $F$ with the desired capacity (for completeness, details are presented in Lemma \ref{ovvio} below).
\end{proof}

One might think that the splitting process could be continued for an infinite time, producing a 
stronger result. This is not the case: if one does not stop the procedure, the set $E$ ``fades away'' and it will have
null capacity, as Theorem \ref{positiveT} below shows.

It is also possible to prove, using an easy convexity argument, a quantitative, positive result
justifying Conjecture \ref{conjecture}.
\begin{theorem}\label{notallisbad} Given a set $E$ with
$\cpc^T(E)=\epsilon$, one has the estimate:
$$
\inf\left\{\cpc^T_n(E):\ \cpc^T(E)=\epsilon\right\}
=2^ne_n=\frac{\epsilon}{1-(2-2^{1-n})\epsilon}.
$$
\end{theorem}
The theorem's statement is more expressive if we replace $\epsilon=\cpc^T(\partial T)-\delta
=1/2-\delta$. The estimate becomes
\begin{equation}\label{expressive}
\inf\left\{\cpc^T_n(E):\ \cpc^T(E)=\epsilon\right\}=
\frac{1/2-\delta}{(2-2^{1-n})\delta+2^{-n}}:
\end{equation}
the lower bound roughly doubles each time $n$ increases by one, until $2^{-n}$ (the ``Euclidean distance''
between the plates of the condenser) reaches the scale of 
$\delta$; after that point, it stabilizes. The difficulty in transfering this result to 
the continuous case consists in the fact that the scale is the amount by which $E$ fails to have 
full capacity. This quantity, to the best of my knowledge, has never been investigated in depth: 
most applications 
involve estimates for sets having ``small enough'' capacity.

\begin{proof}
Let $\psi:[0,1/2]\to[0,1/2]$ be the function
$$
\psi(t)=\frac{t}{2(1-t)}.
$$
The function $\psi$ is a continuous, increasing, strictly  convex diffeomorphism of $[0,1/2]$
onto itself. Let $E$ be a fixed, closed subset of $\partial T$, and, for $x$ in $T$, 
let  $c(x):=\cpc^{T_x}(E\cap\partial T_x)$ be the capacity in the tree $T_x$ of the portion of $E$ 
lying in $\partial T_x$. The recursion relation for capacities can be written in the form
$$
\frac{c(x_+)+c(x_-)}{2}=\psi(c(x)).
$$
Let $\psi^{\circ n}=\psi\circ\dots\psi$ be the composition of $\psi$ with itself $n$
times. We claim that, for $a$ in $T$ fixed and $n$ positive integer
\begin{equation}\label{patatone}
\frac{1}{2^n}\sum_{x\ge a,\ d(x,a)=n}c(x)\ge\psi^{\circ n}(c(a)).
\end{equation}
We prove this by induction. For $n=1$ (\ref{patatone}) holds with equality by the 
recursion relation. Suppose (\ref{patatone}) holds for $n-1$. Then,
\begin{eqnarray*}
\frac{1}{2^n}\sum_{x\ge a,\ d(x,a)=n}c(x)&=&\frac{1}{2}
\left(\frac{1}{2^{n-1}}\sum_{x\ge a_+,\ d(x,a_+)=n-1}c(x)+\frac{1}{2^{n-1}}\sum_{x\ge a_-,\ d(x,a_-)=n-1}c(x)\right)\crcr
&\ge&\frac{1}{2}\left(\psi^{\circ(n-1)}(c(a_+))+\psi^{\circ(n-1)}(c(a_-))\right)\crcr
&\ge&\psi^{\circ(n-1)}\left(\frac{c(a_+)+c(a_-)}{2}\right)\crcr
&=&\psi^{\circ n}(c(a))\crcr
&=&\frac{c(a)}{2^n-(2^{n+1}-2)c(a)}.
\end{eqnarray*}
The explicit calculation of $\psi^{\circ n}$ can be checked by induction. Set $a=o$ to finish the proof.
\end{proof}

\paragraph{Proof of Theorem \ref{positive} on trees.}
\begin{theorem}\label{positiveT}
If $\cpc^T(E)>0$, then
$$
\lim_{n\to\infty}\cpc^T_n(E)=+\infty.
$$
\end{theorem}
\begin{proof}
Let $h$ be the extremal function for the definition of $\cpc^T(E)$ and let $H=Ih$.
By properties (b) and (f),
$$
0<\cpc^T(E)=\sum_{x:\ d(x)=n}h(x).
$$
Let $x^{-1}$ be the parent of the point $x\in T$.
By Egoroff's Theorem, for all $\delta>0$ there is a set $E_\delta$ s.t. $\mu(E_\delta)<\delta$
and $1-H(x^{-1})\to0$ uniformly as $x\to\zeta\in\partial T\setminus E_\delta$. Here, Egoroff's 
Theorem is applied to the sequence of functions $H_n:\partial T\to\RR,$ $H_n(\zeta)=H(x)$ if
$d(x)=n$ and $\zeta\in \partial S(x)$.
By regularity of the measure $\mu$, doubling $\delta$,
we can assume that $E_\delta$ is open; i.e. 
it is union of ``arcs'' of the form $\partial S(y)$.

By rescaling, 
it is easy to see that 
$$
h(x)=(1-H(x^{-1}))\cdot \cpc^{T_x}(E_x).
$$ 
Let $h^x$ be the extremal function for $E_x$ in $T_x$. Then, $h^x$ satisfies the additivity 
relation (b) in $T_x$ and $\sum_{y\in P(\zeta)\setminus P(x^{-1})}h^x(y)\ge1$ for nearly all $\zeta$ in $E_x$.
An obvious candidate is $h^x=(1-H(x^{-1}))\cdot h$ and it is easy to see that such guess has the minimizing 
property of the desired extremal function. By (f), 
$$
\cpc^{T_x}(E_x)=h^x(x)=(1-H(x^{-1}))\cdot h(x).
$$
Since $H_n$ converges uniformly on $\partial T\setminus E_\delta$, there is $n(\delta)$ s.t., 
for $n\ge n(\delta)$ we have
$1-H(x^{-1})=1-H_n(\zeta)\le\delta$ if $d(x^{-1})=n$ and $\zeta\in \partial S(x)\cap\partial T\setminus E_\delta$.

Putting all this together, 
with $n\ge n(\delta)$,
\begin{eqnarray*}
0&<&\cpc^T(E)=\sum_{x:\ d(x)=n}h(x)\crcr
&=&\sum_{d(x)=n,\ \partial S(x)\cap (\partial T\setminus E_\delta)\ne\emptyset}(1-H(x^{-1}))\cdot \cpc^{T_x}(E_x)
+\sum_{d(x)=n,\ \partial S(x)\subset E_\delta}I^*\mu(x)\crcr
&\le&\sum_{d(x)=n,\ \partial S(x)\cap (\partial T\setminus E_\delta)\ne\emptyset}(1-H(x^{-1}))\cdot \cpc^{T_x}(E_x)
+\mu(E_\delta)\crcr
&\le&\delta\sum_{d(x)=n,\ \partial S(x)\cap (\partial T\setminus E_\delta)\ne\emptyset} 
\cpc^{T_x}(E_x)+\delta.
\end{eqnarray*}
Thus, 
$$
0<\cpc^T(E)\le \delta\sum_{d(x)=n} \cpc^{T_x}(E_x)+\delta=\delta(\cpc^T_n(E)+1),
$$
and the result follows letting $\delta\to0$.
\end{proof}
\section{Continuous capacities vs. discrete capacities.}
The usual dyadic decomposition of the unit disc can be thought of a as tree structure $T$ 
(as it is explained below). The boundary of the unit disc can be thought of as the boundary
$\partial T$ of 
the tree (this involves some technicalities, which are especially easy in our case, since the unit 
circle is topologically one-dimensional).

The following theorem is proved in \cite{BePe}. A proof which applies to a more general 
case is in \cite{ARSW}.
\begin{theorem}\label{theorema}
Let $E$ be a closed subset of $\partial T$, identified with a closed subset of $\TT$. Then,
$$
\cpc(E)\approx\cpc^T(E).
$$
\end{theorem}
In this section, we prove a similar result for condenser capacities. For $r=1-2^{-n}$, let 
$$
\cpc_n:=\cpc(E,\DDo(0,r)).
$$
\begin{theorem}\label{comparison}
If $E$ is a closed subset of $\TT$, identified with a closed subset of $\partial T$, then
$$
\cpc_n(E)\approx\cpc^T_n(E).
$$
\end{theorem}
\paragraph{The dyadic decomposition of the disc.} For integers $n\ge0$ and $1\le j\le 2^n$,
consider the \it Bergman box \rm
\begin{equation}\label{Whitney}
Q(n,j)=\left\{z=re^{i\theta}\in\DD:\ \frac{1}{2^{n+1}}<1-r\le\frac{1}{2^n},
\ \frac{j-1}{2^n}\le\frac{\theta}{2\pi}<\frac{j}{2^n}\right\},
\end{equation}
and let $\TTT=\left\{(n,j):\ n\ge0,\ 1\le j\le 2^n\right\}$ be the set of such boxes. We associate
to each $Q=Q(n,j)$ in $\TTT$: a distinguished point $z(Q)$ in $Q$, 
$$
z(Q)=(1-2^{-n-1/2})e^{i\frac{j-1/2}{2^n}};
$$ 
a \it Carleson box \rm 
$$
S(Q)=\left\{z=re^{i\theta}\in\DD:\ 0<1-r\le\frac{1}{2^n},
\ \frac{j-1}{2^n}\le\frac{\theta}{2\pi}<\frac{j}{2^n}\right\};
$$
and a distinguished boundary arc $I(Q)$ in $\TT$, 
$$
I(Q)=\left\{e^{i\theta}\in\DD:
\ \frac{j-1}{2^n}\le\frac{\theta}{2\pi}<\frac{j}{2^n}\right\}.
$$
We will freely use obvious variations on the notation just introduced. For instance, we write $I(n,j)=I(Q)$ when $Q=Q(n,j)$. Also, we might write $Q=Q(I)$ if $I=I(Q)$. Etcetera.

\paragraph{The tree structure.} The set $\TTT$ is given a \it tree structure, \rm which will be denoted by the same letter $\TTT$. The points of $\TTT$ are the \it vertices. \rm 
There is an \it edge \rm of the tree between $(n.j)$ and $(m,i)$ if $n=m+1$ and $I(n,j)\subseteq I(m,i)$ ($I(n,j)$ is one of the two halves of the arc $I(m,i)$) or, viceversa, if $m=n+1$ and $I(m,i)\subseteq I(n,j)$. 
The \it level \rm of the box $Q=Q(n,j)$ is $d_\TTT(Q):=n$; so that $I(Q)=2^{-d_\TTT(Q)}$. 
Note that there is just one vertex $o:=(0,1)$ having level $d_\TTT(o)=0$: it is the \it root \rm of the tree $\TTT$.
Boxes and labels for boxes are sometimes identified: $Q(n,j)\equiv(n,j)$.

We begin with the easy inequality in Theorem \ref{comparison}.
\begin{lemma}\label{easypiece}
$$
\cpc^T_n(E)\lesssim\cpc_n(E).
$$
\end{lemma}
\begin{proof}
Consider the subtrees $T_x$ of $T$, $d(x)=n$, viewed as trees of Bergman boxes, as above. For each 
$\alpha$ in $T_x$, let $z(\alpha)$ be the center of the box $Q(\alpha)$ in $\DD$. Let $\varphi$ be the 
extremal function for the definition of $\cpc_n(E)$ and define a function $h:T\to\RR$ by
$$
h(\alpha):=\varphi(z(\alpha))-\varphi(z(\alpha^{-1})).
$$
It is clear that $h(\beta)=0$ for $d(\beta)\le n-1$ and that
$$
\sum_{\gamma=x}^\alpha h(\gamma)=\varphi(z(\alpha))\ \forall \alpha \in T_x.
$$
Estimating differences $h(\alpha):=\varphi(z(\alpha))-\varphi(z(\alpha^{-1}))$ 
and integrating, we see that
\begin{equation}\label{giacomo}
\|h\|_{\ell^2(T)}^2\lesssim\|\nabla\varphi\|_{L^2(\DD)}^2.
\end{equation}
In fact, $\varphi$ is harmonic in the annulus $\{re^{i\theta}:\ 0<1-r\le 2^{-n}\}$, hence
\begin{eqnarray*}
&&|\varphi(z(\alpha))-\varphi(z(\alpha^{-1}))|\crcr
&&=\int_{z(\alpha^{-1})}^{z(\alpha)}\nabla\varphi(w)\cdot dw\crcr
&&\lesssim(1-|z(\alpha)|)|\nabla\varphi(w(\alpha))|\crcr
&&\text{for\ some\ $w(\alpha)$\ in\ the\ closure\ of\ }\ Q(\alpha)\cup Q(\alpha^{-1})\crcr
&&=(1-|z(\alpha)|)\left|\frac{1}{|B_\alpha|}\int_{B_\alpha}\nabla\varphi(w)dA(w)\right|\crcr
&&\text{by the Mean Value Property,}\crcr
&&\text{where\ $dA$\ is\ area\ measure\ and\ $B_\alpha$\ is\ a\ small\ disc\ centered\ at\ $w(\alpha)$}\crcr
&&\text{having\ radius\ and\ distance\ from\ $\TT$\ comparable\ to\ }(1-|z(\alpha)|)\crcr
&&\lesssim \left(\int_{B_\alpha}|\nabla\varphi(w)^^2dA(w)\right)^{1/2}\crcr
&&\text{by\ Jensen's inequality.}
\end{eqnarray*}
Estimate (\ref{giacomo}) follows, since the discs $B_\alpha$ have bounded overlapping.

On the other hand, as $\alpha\to\zeta\in\partial T$ in $T$, $z(\alpha)\to\Lambda(\zeta)$, the 
image of $\zeta$ in $\TT$, nontangentially. In turn, this implies that
$$
Ih(\alpha)=\varphi(z(\alpha))\to1,
$$
but for a set of null capacity in $\partial T$ (actually, the preimage of a set on null capacity in
$\TT$; but by Theorem A this is the same as null capacity in $\partial T$).

Then, $h$ is admissible for the definition of tree capacity $E$; hence (\ref{giacomo}) 
implies the lemma.
\end{proof}

We now come to the more difficult inequality in Theorem \ref{comparison},
\begin{equation}\label{lesseasy}
\cpc_n(E)\lesssim\cpc_n^T(E).
\end{equation}
 We start with a localization lemma for the 
condenser capacity.

Fix integer $n\ge1$, large enough, and let
$E_j=E\cap I_{n,j}$, where $I_{n,j}$ ($1\le j\le2^n$) is the dyadic arc on $\TT$ defined before.
Let $A_n=\Delta\setminus\overline{\Delta(0,1-2^{-n})}$ be the annulus and let $R\subset A_n$
be the curvilinear rectangle
$$
R=\left\{re^{it}\in A_n:\ \frac{-2}{2^n}\le\frac{t}{2\pi}\le\frac{3}{2^n}\right\}.
$$ 
and let $I^\prime_R=\partial R\cap\partial\Delta(0,1-2^{-n})$ be the side of $R$ which is closest to 
the center of $\Delta$. We also need $I_R$, the union of $I^\prime_R$ and of the parts of $\partial R$
lying on the radii $\frac{t}{2\pi}=\frac{-2}{2^n}$ and $\frac{t}{2\pi}=\frac{3}{2^n}$.  Define
$$
\cpc_R(I_R^\prime,E_0)=\inf\left\{\|\nabla\varphi\|_{L^2(R)}^2:\ \varphi|_{I_R}=0,\ \varphi|_{E_0}\ge1\right\}
$$
to be the capacity of the condenser $(I_R^\prime,E_0)$ in $R$.
\begin{lemma}\label{cutandpaste}
$$
\cpc_R(I_R,E_0)\lesssim\cpc_R(I_R^\prime,E_0).
$$
\end{lemma}
By trivial comparison, the opposite inequality $\cpc_R(I_R,E_0)\ge\cpc_R(I_R^\prime,E_0)$ holds. 
\begin{proof}
To prove the lemma, we use a cut-off argument. Let $\chi$ be a smooth cutoff function on $A_n$:
$$\chi(re^{it})=
\begin{cases}
1&\mbox{if}\ \frac{-1}{2^n}\le\frac{t}{2\pi}\le\frac{2}{2^n};\crcr
0&\mbox{if}\ \frac{t}{2\pi}\le\frac{-2}{2^n}\ \mbox{or}\ \frac{t}{2\pi}\ge\frac{3}{2^n}.
\end{cases}
$$
We can choose $\chi$ in such a way that $0\le\chi\le1$ on $A_n$ and that 
$$
\|\nabla\chi\|_{L^2(A_n)}^2\approx1.
$$
Let $\varphi$ be the extremal function for $\cpc_R(I_R^\prime,E_0)$. Then, $\varphi\cdot\chi$ is an 
admissible function for $\cpc_R(I_R^\prime,E_0)$. It suffices, then, to prove that

\smallskip

\noindent\bf Claim. \rm $\|\nabla(\varphi\cdot\chi)\|_{L^2(A_n)}^2\lesssim\cpc_R(I_R^\prime,E_0)$.

\smallskip

\noindent We have $\|\nabla(\varphi\cdot\chi)\|_{L^2(A_n)}^2\lesssim
\|\chi\nabla \varphi \|_{L^2(A_n)}^2+
\|\varphi\nabla \chi \|_{L^2(A_n)}^2=I+II$. The first summand is o.k.:
$I\le\|\nabla\varphi\|_{L^2(R)}^2=\cpc_R(I_R^\prime,E_0).$ About the second, the integrand is supported
in 
$$
Q=\left\{re^{it}\in A_n:\ 
\frac{-2}{2^n}\le \frac{t}{2\pi}\le
\frac{-1}{2^n}\right\}\cup\left\{re^{it}\in A_n:\  \frac{2}{2^n}\le\frac{t}{2\pi}\le\frac{3}{2^n}
\right\}
$$ 
and we are done if we show that
$$
M^2:=\sup_{z\in Q}|\varphi(z)|^2\lesssim \cpc_R(I_R,E_0).
$$
Let $K:=\left\{z\in R:\ \varphi(z)\ge M/2\right\}=\sqcup_j K_j$, where each $K_j$ is a connected component
of $K$: $K_j$ is closed in $R$ and its closure in the plane meets the boundary of $R$, by the maximum principle 
($\varphi$, being extremal, is harmonic in $R$). Let $K^\prime_j$ be a components of $K$ having a point in $Q$ 
and having nonempty interior (there must be one, by definition of $M$ and by continuity of $\varphi$).
If the closure of $K_j^\prime$ does not  meet $I_{n,0}$, the arc containing $E_0$,
we can replace $\varphi$ by $M/2$ on $K_j^\prime$, strictly reducing the Dirichlet integral of $\varphi$
on $R$, which contradicts the extremality of $\varphi$. Then, there is a continuum $K_j^\prime$
joining a point $z_0$ in $Q$ and a point $z^\prime$ in 
$Q_0=\left\{re^{it}\in A_n:\ 0\le \frac{t}{2\pi}\le\frac{1}{2^n}\right\}$ on which $\varphi\ge M/2$. 
Let 
$$
Q^\prime=\left\{re^{it}\in A_n:\ 
\frac{-1}{2^n}\le \frac{t}{2\pi}\le0\right\}\cup
\left\{re^{it}\in A_n:\  \frac{1}{2^n}\le\frac{t}{2\pi}\le\frac{2}{2^n}\right\}
$$ 
and let $I^\prime_1=\partial Q^\prime\cap\partial\Delta(0,1-2^{-n})$,
 $I^\prime_2=\partial Q^\prime\cap\partial\Delta(0,1)$.
Obvious 
comparison shows that 
\begin{eqnarray}
1&\approx&\cpc_{Q^\prime}(I^\prime_1,I^\prime_2)\crcr
&\le&\left\|\nabla\left(\frac{\varphi}{M/2}\right)\right\|_{L^2(Q^\prime)}^2\crcr
&&\text{because the function }\frac{\varphi}{M/2}\text{ is admissible for the condenser capacity}\crcr
&\le&\frac{4}{M^2}\|\nabla\varphi\|_{L^2(R)}^2\crcr
&\le&\frac{4}{M^2}\cpc_R(E_0).
\end{eqnarray}
i.e., $M^2\lesssim\cpc_R(E_0)$, as wished.
\end{proof}

We now come to the proof of (\ref{lesseasy}).

Let $R_j$ be a rectangle as $R$, but built starting from the set $E_j$. Let 
$$
E^{(k)}=\sqcup_{j=5n+k}E_j.
$$
Since the sum of the extremal functions for the five pieces of $E$ is admissible ffor $E$,
\begin{equation}\label{tuttiacasa}
\cpc_n(E)\le 5\sum_{k=0}^4\cpc_n(E^{(k)}).
\end{equation}
Also, by comparison:
\begin{equation}\label{servira}
\cpc_n(E^{(k)})\le\sum_n\cpc_R(I_{5n+k},E_{5n+k}).
\end{equation}
In fact, if $\varphi_n$ are extremal functions for $\cpc_{R_{5n+k}}(I_{R_{5n+k}},E_{5n+k})$,
extended to be zero in $A_n\setminus R_{5n+k}$, then
$$
\varphi=\sum_{n}\chi_{R_{5n+k}}\varphi_n
$$
is admissible for $\cpc_n(E^{(k)})$ and 
$\|\nabla\varphi\|_{L^2}^2=\sum_n\|\varphi_n\|_{L^2(R_{5n+k})}^2$. The inequality follows by 
definition of capacity.

By (\ref{tuttiacasa}), (\ref{servira}) and Lemma \ref{cutandpaste}, then:
\begin{equation}\label{parziale}
\cpc_n(E)\lesssim \sum_{k=0}^4\sum_n\cpc_{R_{5n+k}}(I_{R^\prime_{5n+k}}^\prime,E_{5n+k})=
\sum_l\cpc_{R_{l}}(I_{R^\prime_{l}}^\prime,E_{l}).
\end{equation}

\smallskip

The quantity $\cpc_{R_{l}}(I_{R^\prime_{l}}^\prime,E_{l})$ 
verifies the condition under which capacity can be discretized
as in \cite{BePe} or \cite{ARSW}. In fact, the proof of Theorem \ref{theorema} can be adapted without 
changes to show that
$$
\cpc_{R_{l}}(I_{R^\prime_{l}}^\prime,E_{l})\approx \cpc^{T_{x_l}}(E_j).
$$
Summing over $l$ and using \it additivity \rm of these special capacities in the tree $T$,
\begin{eqnarray*}
\cpc_n(E)&\lesssim&\sum_l\cpc^{T_{x_l}}(E_j)\crcr
&=&\cpc^T_n(E),
\end{eqnarray*}
as wished. The proof of Theorem \ref{comparison} is ended.
\paragraph{Proofs of the main theorems.}
\noindent{\bf Proof of Theorem \ref{positive}.} Since $r\mapsto\cpc(E,\DDo(0,r))$ is increasing,
it suffices to test the conclusion of the theorem on $r=1-2^{-n}$, for integer $n$. By Theorem 
\ref{comparison},
$$
\cpc(E,\DDo(0,1-2^{-n}))\gtrsim\cpc^T_n(E)\to\infty
$$ 
as $n\to\infty,$ by Theorem \ref{positiveT}.

\medskip

\noindent{\bf Proof of Theorem \ref{negative}.} If $\epsilon>0$ is small enough, then, 
by Theorem \ref{comparison} (rather, by the special case proved in \cite{BePe} and \cite{ARSW}),
if $\cpc(E)\le\epsilon$, then $0<\cpc^T(E)\le\epsilon^\prime<\cpc^T_n(\partial T)=1/2$. 
By Theorem \ref{negativeT}, there is $R(\epsilon)$
s.t. for all $n$ there is $E$ with $\cpc^T(E)\le\epsilon^\prime$ and $\cpc_n(E)\le R(\epsilon)$.
By Theorem \ref{comparison}, this implies Theorem \ref{negative}.

\medskip
We finish with the proof of a Lemma used in the proof of Theorem \ref{negative}.
\begin{lemma}\label{ovvio}
For each $0\le e\le1/2$ there is a closed subset $E$ of $\partial T$ such that $\cpc^T(E)=e$.
\end{lemma}
\begin{proof}
Let $\Lambda:\partial T\to[0,1]$ be the map associating to a geodesic $\zeta$ in $\partial T$, 
$P(\zeta)=\{\zeta_k:\ k\ge0\}$ being an enumeration ot its vertices where $d(\zeta_n)=n$, 
the point $t$ in $[0,1]$ such that
$$
e^{2\pi it}=\cap_{k\ge0}\overline{I(Q(k,j))}.
$$
We assume that the geodesic ``to the extreme left'' maps to $0$, while that to the ``extreme right'' maps to $1$.

It is easy to prove that the map $\Lambda$ is continuous (in fact, Lipschitz) w.r.t. the metrics $\rho$ on $\partial T$
and Euclidean on $[0,1]$. Define a function $f:[0,1]\to[0,1/2]$ by
$$
f(t)=\cpc^T(\Lambda^{-1}([0,t])).
$$
Clearly $f(0)=0$, $f(1)=1/2$ and $f$ increases. It suffices to prove that $f$ is continuous.

We have the inequalities (for $h>0$):
\begin{eqnarray*}
f(t)&\le& f(t+h)\crcr
&=&\cpc^T(\Lambda^{-1}([0,t+h]))\le \cpc^T(\Lambda^{-1}([0,t]))+\cpc^T(\Lambda^{-1}([t,t+h]))\crcr
&&\text{by subadditivity of capacity}\crcr
&=&f(t)+o_{h\to0}(1),
\end{eqnarray*}
by regularity of capacity: $\lim_{h\to0}\cpc^T(\Lambda^{-1}([t,t+h]))=\cpc^T(\Lambda^{-1}([t,t]))=0$. Hence, $f$ is 
right continuous.

Similarly, one shows that $f(t-h)+o_{h\to 0}(1)\ge f(t)$, deducing that $f$ is left continuous.
\end{proof}

\end{document}